\documentclass[12pt,a4paper]{amsart}
\usepackage{amsfonts}

\usepackage{amsmath}
\usepackage{amssymb}

\usepackage{mathtools}
\usepackage{amsthm}
\usepackage{enumerate}

\usepackage[small,nohug,heads=littlevee]{diagrams}
\usepackage{hyperref}

\def\redH{\tilde{H}}

\newcommand{\F}{{\mathbb{F}}}
\newcommand{\N}{{\mathbb{N}}}
\newcommand{\R}{\mathbb{R}}
\newcommand{\Q}{\mathbb{Q}}
\newcommand{\Z}{{\mathbb{Z}}}

\newcommand{\orbSp}[2]{#1/{#2}}
\newcommand{\inv}[2]{#1^{#2}}
\newcommand{\redBor}[2]{#1\rtimes_{#2} E#2}
\newcommand{\bor}[2]{#1 \times_{#2} E#2}
\newcommand{\stuntRedBor}[2]{#1\rtimes_{#2} E_\infty^1 #2}
\newcommand{\stuntBor}[2]{#1\times_{#2} E_\infty^1 #2}

\def\Loop{\Omega}

\def\pt{{\mathrm{pt}}}

\newcommand{\RedCoprod}{\bigvee}
\def\union{\cup}

\newcommand{\redProd}{\wedge}

\newcommand{\sus}{\Sigma}
\newcommand{\onto}{\twoheadrightarrow}
\newcommand{\inc}{\hookrightarrow}

\newcommand{\proj}{\mathop{\mathrm{proj}}}

\def\isom{\cong}

\newcommand{\DirSum}{\bigoplus}

\newcommand{\ivBrack}[1]{[#1]}
\newcommand{\redBetti}[2]{\tilde{b}_{#2}(#1)}

\newcommand{\redPoinSer}[1]{\tilde{P}\! \left(#1\right)}

\newcommand{\tilDelta}{\tilde{\Delta}}

\newcommand{\projSp}[2]{#2 {\mathbb{P}}^{#1}}
\newcommand{\redDiag}[1]{\overline{\Delta}_{#1}}

\newcommand{\ind}{x}

\newcommand{\multInd}{\N^{<\omega}}
\renewcommand{\dim}[1]{\mathop{\mathrm{dim}} #1}
\newcommand{\length}[1]{|#1|}

\newenvironment{nouppercase}{%
  \renewcommand{\uppercasenonmath}[1]{}}{}

\newarrow{Equals}{=}{=}{=}{=}{=}
\newarrow{Inc}{C}{-}{-}{-}{>}

\newtheorem{theorem}{Theorem}
\numberwithin{theorem}{section}
\newtheorem{lemma}[theorem]{Lemma}
\newtheorem{proposition}[theorem]{Proposition}
\newtheorem{corollary}[theorem]{Corollary}

\theoremstyle{definition}

\newtheorem{example}[theorem]{Example}

\begin{document}

\title[Loop homological invariants]{Loop homological invariants \\ associated to real projective spaces}
\author{Man Gao \and Colin Tan \and Jie Wu}
\date{}
\address{Man Gao,
3A Happy Avenue North, Happy Mansion, Singapore 369741}
\email{gao2man@nus.edu.sg}
\address{Colin Tan, Department of Mathematics, National University of Singapore, Block
S17, 10 Lower Kent Ridge Road, Singapore 119076}
\email{colinwytan@gmail.com}
\address{Jie Wu, Department of Mathematics, National University of Singapore, Block
S17, 10 Lower Kent Ridge Road, Singapore 119076}
\email{matwuj@nus.edu.sg}

\thanks{
Jie Wu was partially supported by the research grants Tier 1 WBS No. R-146-000-190-112
    and NSFC of China (grant number 11329101).
}

\begin{nouppercase}
\maketitle
\end{nouppercase}

\begin{abstract}
Let $A$ be a based subspace of $Y$.
Under the assumptions that
        $Y$ is path-connected
        and that the reduced diagonal map of $A$
            induces the zero map in all mod 2 reduced homology groups,
    we compute a formula for the mod 2 reduced Poincar\'{e} series of the loop space
        $\Loop ((A\redProd \projSp{\infty}{\R}) \union_{A\redProd \projSp{1}{\R}} (Y\redProd \projSp{1}{\R}))$.
    Here $\projSp{\infty}{\R}$ and $\projSp{1}{\R}$
        denote the infinite real projective space and the real projective line respectively.
\end{abstract}



\section{Introduction}

A general problem in algebraic topology is to compute the homology of a loop space.
More precisely, for coefficients in a field $K$ and given a pointed space $X$,
                        one can ask to determine the homology $H_*(\Loop X; K)$ as a Hopf algebra.
Here the multiplication of $H_*(\Loop X; K)$ is induced by multiplication of loops $\mu: \Loop X\times \Loop X\to \Loop X$,
    while the comultiplication of $H_*(\Loop X; K)$ is induced by the diagonal map $\Delta_{\Loop X}: \Loop X\to \Loop X\times \Loop X$.
In the case where $X = \Sigma Y$ is the suspension of a path-connected space $Y$,
    this was determined by Bott-Samelson \cite{BottSamelson54}.
     They proved that $H_*(\Loop \Sigma Y; K)$ is isomorphic as a Hopf algebra to the tensor algebra $T(\redH_*(Y;K))$ of the reduced homology of $Y$,
        with the comultiplication of the tensor algebra determined on generators by the comultiplication of $\redH_*(Y; K)$.

In the case where $K$ is of characteristic zero, say $K$ is the field $\Q$ of rational numbers,
    Milnor-Moore proved that $H_*(\Loop X; \Q)$ is isomorphic as a Hopf algebra
        to the universal enveloping algebra $U(\pi_*(\Loop X) \otimes \Q)$ where the Lie bracket on $\pi_*(\Loop X) \otimes \Q$
            is given by the Samelson product \cite{MilnorMoore65}.
However, the structure of Hopf algebras is less understood when $K$ is of characteristic a prime $p$ (although see \cite{Hubbuck81, Hubbuck82}).
In this article, we will be interested in the case where $p = 2$.

When unable to determine the homology $H_*(\Loop X; K)$ as a Hopf algebra,
     one can forget the multiplication and comultiplication, asking only to compute the reduced Poincar\'{e} series of $\Loop X$.
Recall that, for $W$ a pointed space each of whose homology groups $H_q(W; K)$ are finite-dimensional $K$-vector spaces,
    its $q$th {\emph{reduced Betti number}} $\redBetti{W; K}{q}$ is the dimension of the $K$-vector space $\redH_q(W;K)$
    and its {\emph{reduced Poincar\'{e} series}} is the ordinary generating function of its reduced Betti numbers,
        namely the formal power series $\redPoinSer{W;K} := \sum_{q\ge 0} \ind^q \redBetti{W; K}{q}$.
        For example, the Bott-Samelson theorem described above implies that,
            if $Y$ is a path-connected space each of whose homology groups $\redH_q(Y;K)$ are finite-dimensional, then
\begin{equation} \label{eq: redPoinSerOfLoopSus}
         \redPoinSer{\Loop\sus Y;K} = \frac{\redPoinSer{Y;K}}{1-\redPoinSer{Y;K}}.
         \end{equation}

For $K$ of arbitrary characteristic, a standard strategy to compute the loop space homology $H_*(\Loop X; K)$ is
    to consider the Serre spectral sequence \cite{Serre51} and the Eilenberg-Moore spectral sequence \cite{EilbenbergMoore66}
                associated to the path-loop fibration $\Loop X \to P X \to X$.

Another strategy is to construct a topological monoid whose underlying space is $\Loop X$.
Again considering the example where $X = \sus Y$ is the suspension of a path-connected space $Y$,
                James proved that the reduced free topological monoid $J[Y]$ has homotopy type $\Loop\Sigma Y$
                    and used the associated word filtration to prove the suspension splitting $\Sigma \Loop \Sigma Y \simeq \RedCoprod_{s\ge 1} \Sigma Y^{\redProd s}$ \cite{James55}.
This gives another proof of \eqref{eq: redPoinSerOfLoopSus}.
The idea behind this strategy is to exploit the strictly associative multiplication of the constructed topological monoid.
This strict associativity is easier to exploit that the homotopy coherent associativity of the multiplication of loops $\mu: \Loop X\times \Loop X\to\Loop$
    which gives $\Loop X$ its $A_\infty$-space structure \cite{Stasheff63}.

In addition, one can apply techniques from simplicial homotopy theory.
For $X$ a reduced simplicial set,
    Kan constructed a free simplicial group $GX$ whose underlying space is $\Loop X$ \cite{Kan58}.
Taking $K$ to be the field $\F_2$ of two elements, Bousfield-Curtis used Kan's construction to develop a spectral sequence
    which converges $H_*(\Loop X; \F_2)$ when $X$ is simply connected \cite{BousfieldCurtis70}.
A consequence of their work is the following result (see proposition \ref{prop: BousfieldCurtis} below): If $X$ is a simply-connected pointed space whose
reduced diagonal map $\redDiag{X} : X\to X\redProd X$ induces the zero map in all mod 2 reduced homology groups, then
\begin{equation} \label{eq: BousfieldCurtisFirstEq}
\redPoinSer{\Loop X ; \F_2} = \frac{\redPoinSer{X; \F_2}}{\ind - \redPoinSer{X;\F_2}}
\end{equation}
Here the reduced diagonal map $\redDiag{X}$ is the composite $X \xrightarrow{\Delta_X} X\times X\onto X\redProd X$ of the diagonal map followed by the standard projection to the self-smash product.
In particular, \eqref{eq: BousfieldCurtisFirstEq} holds when $X$ is a simply-connected co-H-space.
This is a generalization of \eqref{eq: redPoinSerOfLoopSus} in the case where $K = \F_2$.

In this article, we compute the mod 2 reduced Poincar\'{e} series
    for a certain loop space which is the underlying space of a simplicial group construction of Carlsson.
This culminates work beginning from Carlsson \cite{Carlsson84} and followed by the first and third authors \cite{Wu97, GaoThesis, GaoWu13}.
Let $\projSp{1}{\R}$ denote the real projective line,
    regarded as a subspace of the infinite real projective space $\projSp{\infty}{\R}$.
    In terms of the standard CW complex structure on $\projSp{\infty}{\R}$,
        the subspace $\projSp{1}{\R}$ is the bottom cell.
Note that both these real projective spaces are Eilenberg-MacLane spaces,
        namely $\projSp{1}{\R}= K(\Z,1)$ and $\projSp{\infty}{\R }=K(\Z/2 , 1)$.
In particular, $\projSp{1}{\R} \simeq S^1$.
\begin{theorem} \label{thm: PoinSer}
Let $A\inc Y$ be a based inclusion of pointed spaces,
    both of whose mod 2 homology groups are finite-dimensional.
If $Y$ is path-connected and
the map $(\redDiag{A})_* : \redH_*(A;\F_2)\to \redH_*(A\redProd A;\F_2)$ in mod 2 reduced homology induced by the reduced diagonal map of $A$ is the zero map,
  then
  \begin{multline} \label{eq: PoinSerFormula}
  \redPoinSer{\Loop((A\redProd \projSp{\infty}{\R}) \union_{A\redProd \projSp{1}{\R}} (Y\redProd \projSp{1}{\R}) ; \F_2)} \\
    = \frac{                (1-\ind)\redPoinSer{Y;\F_2} + \ind \redPoinSer{A;\F_2} }
              {   1 - \ind - (1-\ind)\redPoinSer{Y;\F_2} - \ind\redPoinSer{A;\F_2}  }
  \end{multline}
\end{theorem}

For $G$ a discrete group and $X$ a pointed $G$-space,
    Carlsson constructed a simplicial group $J^G[X]$ of the homotopy type of $\Loop \bar{C}(X)$,
        where $\bar{C}(X)$ is the homotopy cofiber of composite $X\inc \bor{X}{G} \onto \redBor{X}{G}$,
            where $\bor{X}{G}$ (resp. $\redBor{X}{G}$) is the Borel construction (resp. reduced Borel construction) \cite{Carlsson84}.
The first and third authors named $\bar{C}(X)$ the {\emph{1-stunted reduced Borel construction}} of $X$ \cite{GaoWu13}.
In the case where $G= C_2$ is the cyclic group of order two,
    they used Carlsson's construction to obtain a homology decomposition (see \eqref{eq: E1 collapse} below).
The space $(A\redProd \projSp{\infty}{\R}) \union_{A\redProd \projSp{1}{\R}} (Y\redProd \projSp{1}{\R})$ in \eqref{eq: PoinSerFormula}, which is the union
    of the smash products $Y\redProd \projSp{1}{\R}$ and $A\redProd\projSp{\infty}{\R}$
            identified over their common subspace $A\redProd\projSp{1}{\R}$,
is the 1-stunted reduced Borel construction of the pointed $C_2$-space $Y\union_A Y$
    with the $C_2$-action associated to the based involution switching the two copies of $Y$ (see lemma \ref{lem: unionAsStuntRedBor} below).

Theorem \ref{thm: PoinSer} generalizes \eqref{eq: redPoinSerOfLoopSus} in the case where $K= \F_2$ by taking $A$ to be the basepoint of $Y$ (see example \ref{ex: loopSus} below).
Further examples are given in section \ref{sec: examples}.
We highlight example \ref{ex: LangExample} which has some relation to combinatorics.
Taking $A = S^1$ and $Y = S^2$ in Theorem \ref{thm: PoinSer}
 gives a loop space whose mod 2 reduced Betti numbers form essentially sequence A052547 in the On-Line Encyclopedia of Integer Sequences \cite{OEIS}.
This sequence has a geometric interpretation in terms of diagonals lengths in the regular
heptagon with unit side length \cite{Steinbach97,Lang12}. These diagonal lengths are related
to the Chebyshev polynomials used in approximation theory.

This article is organized as follows.
In section \ref{sec: BackgroundMaterial}, we cover background material taken mainly from \cite{GaoWu13}.
In section \ref{sec: ProofOfPoinSerFormula}, we prove theorem \ref{thm: PoinSer}.
In section \ref{sec: examples}, we give some examples of theorem \ref{thm: PoinSer} and present a corollary related to the above-mentioned spectral sequence of Bousfield-Curtis.

\section{Background material} \label{sec: BackgroundMaterial}

In this section, we cover background material taken mainly from \cite{GaoWu13}.
The results will be stated without proof, but the relevant reference will be indicated.
Throughout the rest of this article, homology will always be taken modulo 2.
As such, we omit the base field $\F_2$ from the notation.

First we introduce the 1-stunted reduced Borel construction associated to a pointed $G$-space, where $G$ is a discrete group.
    A {\emph{(right) $G$-space}} is a space $X$ equipped with a map $\mu: X\times G\to X$
        such that, for $x\in X$ and $g,h\in G$, we have the identities $\mu(x,1) = x$ and $\mu(\mu(x,g),h)=\mu(x,gh)$.
    Typically, we write $x\cdot g$ or just $xg$ instead of $\mu(x,g)$.
    For $X$ a $G$-space,
    let $\inv{X}{G}$ denote its $G$-invariant subspace
    and let $\orbSp{X}{G}$ denote its orbit space.

    The terminal $G$-space is the one-point space, denoted by $\pt$, equipped with the trivial $G$-action.
    Hence, a {\emph{pointed $G$-space}}
            is a $G$-space $X$
                equipped with a $G$-equivariant map from $\pt$ to $X$.
    Equivalently,
        a pointed $G$-space
                is a $G$-space equipped with a basepoint invariant under the $G$-action.
    For example, for $Z$ a pointed space, the quotient space
    \begin{equation} \label{eq: RedProdSpace}
    Z\rtimes G := (Z\times G)/(\pt\times G),
    \end{equation}
    equipped with the action induced by the free $G$-action on $Z\times G$, is a pointed $G$-space.

Let $EG$ denote the contractible $G$-space with a free $G$-action.
    The {\emph{Borel construction}} of a $G$-space $X$,
        denoted by $\bor{X}{G}$,
            is the orbit space of the diagonal $G$-action on the product $X\times EG$.
    The  {\emph{1-stunted Borel construction}} of a $G$-space $X$,
        denoted by $\stuntBor{X}{G}$,
             is the homotopy cofiber of the inclusion of a fiber $X\inc \bor{X}{G}$ into its Borel construction.
    The {\emph{reduced Borel construction}} of a pointed $G$-space $X$,
        denoted by $\redBor{X}{G}$,
            is the homotopy cofiber of the map $\bor{\pt}{G}\to \bor{X}{G}$
                induced by the inclusion $\pt\inc X$.
    Let $i:\bor{X}{G}\to \redBor{X}{G}$ be the natural map.
    The {\emph{1-stunted reduced Borel construction}} of a pointed $G$-space $X$
        denoted by $\stuntRedBor{X}{G}$,
            is the homotopy cofiber of the composite map $X\inc\bor{X}{G} \xrightarrow{i} \redBor{X}{G}$.

These variants of the Borel construction are related by the following $3\times 3$ homotopy commutative diagram.
\begin{diagram}
\pt    & \rTo & \bor{\pt}{G}       & \rTo &\stuntBor{\pt}{G} \\
\dTo &         & \dTo                &        & \dTo \\
X     & \rTo   & \bor{X}{G}       & \rTo  & \stuntBor{X}{G} \\
\dTo &          & \dTo               &           & \dTo \\
X     & \rTo    & \redBor{X}{G} & \rTo & \stuntRedBor{X}{G}
\end{diagram}
    By definition, all three rows and the middle column are cofiber sequences.
Since the first column is obviously a cofiber sequence,
    we conclude that the third column is a cofiber sequence.
In other words, the 1-stunted reduced Borel construction $\stuntRedBor{X}{G}$
    is also the homotopy cofiber of the map $\stuntBor{\pt}{G}\to \stuntBor{X}{G}$ between 1-stunted Borel constructions induced by the inclusion $\pt\inc X$.

Let us consider two examples of the 1-stunted reduced Borel construction, which we leave to the reader to verify.
If $G$ acts trivially on a pointed space $X$, then
\begin{equation} \label{eq: stuntRedBorOfTrivialAction}
\stuntRedBor{X}{G} \simeq X \redProd BG
\end{equation}
Here $BG=K(G,1)$ is the classifying space of the discrete group $G$.
For another example, consider the pointed $G$-space $ Z\rtimes G$ described in \eqref{eq: RedProdSpace} above.
Then there is a homotopy equivalence natural in $Z$:
\begin{equation} \label{eq: stuntRedBorOfSemifreeAction}
\stuntRedBor{(Z\rtimes G)}{G} \simeq Z \redProd \sus G
\end{equation}

Now consider the case where $G= C_2$, the cyclic group of order two.
In this case, a pointed $C_2$-space can be described equivalently as a pointed space equipped with a based involution.
To see this, let $t$ be the generator of $C_2$.
    Given a pointed $C_2$-space $X$, the map $(x\mapsto xt): X\to X$ is a based involution of $X$.
    Conversely, given $j:X\to X$ is a based involution of $X$ (so that $j\circ j = {\mathrm{id}}_X$),
                            then $X$ is a pointed $C_2$-space with the action $X\times C_2 \to X$ given by $x\cdot 1 = x$ and $x\cdot t = j(x)$.

Let $X$ be a $G$-space.
The {\textit{orbit projection}} is the projection $X\to X/G$ which sends each $x\in X$ to its orbit $xG$.
For $f:Y\to Z$ a map, a {\emph{section}} of $f$ is a map $g: Z\to Y$ such that the
composite $Z\xrightarrow{g} Y\xrightarrow{f} Z$ is the identity map of $Z$.
The following homology decomposition is theorem 1.1 of  \cite{GaoWu13} (recall that homology is taken modulo 2):
For $X$ be a pointed $C_2$-space,
if the orbit projection has a section, then there is an isomorphism of $\F_2$-algebras:
\begin{equation} \label{eq: E1 collapse}
\redH_*(\Loop(\stuntRedBor{X}{C_2}))\isom \DirSum_{s\ge 1} \redH_*\left((X/C_2)^{\redProd
s}/\tilDelta_s\right),
\end{equation}
where, for $s\ge 1$,
\[
\tilDelta_s:=\{x_1C_2\redProd \cdots \redProd x_s C_2\in (X/C_2)^{\redProd s}|\, \exists i=1,\ldots, s-1\,
(x_i=x_{i+1}\in \inv{X}{C_2})\}
\]

The sufficient condition for \eqref{eq: E1 collapse} to hold is that $X$ is a pointed $C_2$-space whose orbit projection has a section.
Proposition 4.1 of \cite{GaoWu13} characterizes $C_2$-spaces whose orbit projection has a section:
Let $X$ be a $C_2$-space.
The orbit projection $X\to X/C_2$ has a section if and only if, there exist spaces $A$ and
$Y$ where $A$ is a subspace of $Y$, such that
\begin{equation} \label{eq: chraOfOrbitProjSecExist}
 X\isom Y\union_A Y
 \end{equation}
 with the $C_2$-action
    corresponding to the involution which switches the two copies of $Y$.
For the $C_2$-space $Y\union_A Y$, its orbit space is isomorphic to $Y$ and its $C_2$-invariant subspace is $A$.
There are exactly two sections of the orbit projection.
One section maps the orbit space to the left copy of $Y$, while the other section maps the orbit space to right copy of $Y$.

In the case of a pointed $C_2$-space $X$,
    its basepoint is invariant under the action, hence its orbit projection $X\to X/C_2$ has a section if and only if \eqref{eq: chraOfOrbitProjSecExist} holds where $A$ is a based subspace of $Y$.

Next, we need a combinatorial formula for the Betti numbers of the spaces $\tilDelta_s$.
This will require some notation.
In this article, by a {\emph{multiindex}}, we mean a (possibly empty) finite sequence of positive integers.
        For example, $(2,5,4)$ is a multiindex.
For a multiindex $\alpha = (\alpha_1,\ldots, \alpha_d)$,
     its {\emph{dimension}},
    denoted by $\dim{\alpha}$,
        is just the nonnegative integer $d$,
while its {\emph{length}} $\length{\alpha}$
        is the sum $\alpha_1+\cdots + \alpha_d$.
For $W$ a pointed space whose mod 2 homology groups are finite-dimensional,
    its {\emph{$\alpha$th (mod 2) reduced Betti number}} $\redBetti{W}{\alpha}$
            is the product $\redBetti{W}{\alpha_1} \redBetti{W}{\alpha_2}\cdots \redBetti{W}{\alpha_d}$.
In particular, $\redBetti{W}{\emptyset} = 1$ where $\emptyset$ is the empty sequence.

For a sentence $\tau$, its {\emph{Iverson bracket}} is
\begin{equation} \label{eq: IversonBracketNotation}
\ivBrack{\tau} = \begin{cases}
1 & {\text{if $\tau$ is true}} \\
0 & {\text{if $\tau$ is false}}
\end{cases}
\end{equation}
For $n$ and $k$ integers,
    the {\emph{binomial coefficient}} is given by
        \begin{equation} \label{eq: binomCoeffDef}
        \binom{n}{k} := \ivBrack{k\ge 0} \frac{n(n-1)\cdots (n-k+1)}{k!}
        \end{equation}
    In particular, when $k=0$, the product in the numerator is empty, so that $\binom{n}{0} = 1$.

Using this notation, we can state theorem 1.2 from \cite{GaoWu13}:
Let $s\ge 1$.
If the map $(\redDiag{\inv{X}{C_2}})_* : \redH_*(\inv{X}{C_2})\to \redH_*(\inv{X}{C_2}\redProd \inv{X}{C_2})$
    induced by the reduced diagonal map of $\inv{X}{C_2}$ is the zero map,
then
\begin{equation} \label{eq: tilDeltaDecomp}
\redBetti{\tilDelta_s}{q}
    = \sum_{\substack{|\lambda|+|\mu| = q -s+\dim\lambda+\dim\mu + 1 \\
2\le\dim\lambda+\dim\mu+1\le s}} c_{\lambda,\mu}^{(s)}\redBetti{\orbSp{X}{C_2}}{\lambda} \redBetti{\inv{X}{C_2}}{\mu}
\end{equation}
where
\[
c_{\lambda,\mu}^{(s)} :=\binom{\dim \lambda +\dim\mu}{\dim\mu}\binom{s-\dim \lambda-\dim\mu-1}{\dim\mu-1}.
\]
Here $\lambda$ and $\mu$ in the sum in \eqref{eq: tilDeltaDecomp} are multiindexes.

Finally, iterating corollary 5.7 of  \cite{GaoWu13} yields the following lemma.
\begin{lemma} \label{lem: incAsHomologousToZero}
Let $s\ge 1$. If the map $(\redDiag{\inv{X}{C_2}})_* : \redH_*(\inv{X}{C_2})\to \redH_*(\inv{X}{C_2}\redProd \inv{X}{C_2})$
    induced by the reduced diagonal map of $\inv{X}{C_2}$ is the zero map,
    then the map $\redH_*(\tilDelta_s)\to \redH_*((X/C_2)^{\redProd s})$ induced by the inclusion $\tilDelta_s \inc (X/C_2)^{\redProd s}$ is also the zero map.
\end{lemma}

\section{Proof of theorem \ref{thm: PoinSer}} \label{sec: ProofOfPoinSerFormula}

In this section, we prove theorem \ref{thm: PoinSer}.

First, we compute the homotopy type of the 1-stunted reduced Borel construction of pointed $C_2$-spaces whose orbit projection has a section.
Recall from the remarks after \eqref{eq: chraOfOrbitProjSecExist} that a pointed $C_2$-space whose orbit projection has a section has the form $Y\union_A Y$ where $A$ is a based subspace of $Y$.

\begin{lemma} \label{lem: unionAsStuntRedBor}
For $A\inc Y$ a based inclusion of pointed spaces,
    \[\stuntRedBor{(Y\union_A Y)}{C_2}
                                                    \simeq (A\redProd \projSp{\infty}{\R}) \union_{A\redProd \projSp{1}{\R}} (Y\redProd \projSp{1}{\R})\]
\end{lemma}

\begin{proof}
The following is a pushout square of pointed $C_2$-spaces:
\begin{diagram}
    A \rtimes C_2 & \rTo^{\proj_1} & A\\
    \dInc            &        & \dTo \\
    Y \rtimes C_2   & \rTo & Y\union_A Y
    \end{diagram}

Since the Borel construction commutes with equivariant homotopy colimits
    and the homotopy cofiber commutes with homotopy colimits,
the 1-stunted reduced Borel construction commutes with homotopy colimits of pointed $C_2$-spaces.

Thus,
        taking the 1-stunted reduced Borel construction and using equations \eqref{eq: stuntRedBorOfTrivialAction} and \eqref{eq: stuntRedBorOfSemifreeAction},
         we have a homotopy pushout square
\begin{diagram}
    A\redProd \sus C_2  & \rTo^{A\redProd j} & A\redProd B C_2\\
    \dInc                        &        & \dTo \\
    Y\redProd \sus C_2   & \rTo & \stuntRedBor{(Y\union_A Y)}{C_2}
\end{diagram}
where the map $j: \sus C_2\to BC_2$ can be identified, up to homotopy, with the inclusion $\projSp{1}{\R} \inc  \projSp{\infty}{\R}$.
\end{proof}

Recall from the remarks after \eqref{eq: chraOfOrbitProjSecExist} that,
for the pointed $C_2$-space $Y\union_A Y$,
     its orbit projection has a section, its orbit space is $Y$ and its $C_2$-invariant subspace is $A$.
Hence, lemma \ref{lem: unionAsStuntRedBor} together with the homology decomposition \eqref{eq: E1 collapse} gives
\begin{equation} \label{eq: InitialDecomp}
\redPoinSer{\Omega(A\redProd \projSp{\infty}{\R}) \union_{A\redProd \projSp{1}{\R}} (Y\redProd \projSp{1}{\R})} = \sum_{q\ge 0} \ind^q \sum_{s\ge 1} \redBetti{Y^{\redProd s}/\tilDelta_s}{q},
\end{equation}
where we identify $\tilDelta_s$ with the following subspace of $Y^{\redProd s}$:
\[
\{x_1\redProd \cdots \redProd x_s \in Y^{\redProd s}|\, \exists i=1,\ldots, s-1\,
(x_i=x_{i+1}\in A\}.
\]

To prove theorem \ref{thm: PoinSer}, we will require an ordinary generating function of the binomial coefficients (see \cite{Wilf94} p.120 equation (4.3.1)).
Let $k$ be a nonnegative integer. Then
\[
\sum_{n\ge 0 } \binom{n}{k} \ind^n = \frac{\ind^k}{(1-\ind)^{k+1}}
\]
More generally, given an integer $m$ satisfying $0\le m\le k$, we also have
\begin{equation} \label{eq: StanGenFuncOfBinomCoeff}
\sum_{n\ge m } \binom{n}{k} \ind^n = \frac{\ind^k}{(1-\ind)^{k+1}}
\end{equation}
This is because the definition \eqref{eq: binomCoeffDef} of the binomial coefficients implies that $\binom{0}{k} = \binom{1}{k} = \cdots = \binom{k-1}{k} = 0$.
Note that, by a change of variables, \eqref{eq: StanGenFuncOfBinomCoeff}
    hold yet more generally when the indeterminate $\ind$
        is replaced by a formal power series whose constant term is zero.

\begin{proof}[Proof of  theorem \ref{thm: PoinSer}]
Let $A\inc Y$ be a based inclusion of pointed spaces,
    both of whose mod 2 homology groups are finite-dimensional.
Suppose that $Y$ is path-connected and
    that the map $(\redDiag{A})_* : \redH_*(A)\to \redH_*(A\redProd A)$ in reduced homology induced by the reduced diagonal map of $A$ is the zero map.
Hence, by lemma \ref{lem: incAsHomologousToZero}, for each $s\ge 1$, the map
    $\redH_*(\tilDelta_s)\to \redH_*(Y^{\redProd s})$ induced by the inclusion $\tilDelta_s \inc Y^{\redProd s}$ is also the zero map.
This implies that the long exact sequence in homology associated to the cofiber sequence $\tilDelta_s\inc Y^{\redProd s}\to Y^{\redProd s}/\tilDelta_s$ splits into short exact sequences
    $
    0\to \redH_q(Y^{\redProd s})\to \redH_q(Y^{\redProd s} / \tilDelta_s) \to \redH_{q-1}(\tilDelta_s) \to 0
    $.
As we are taking homology with coefficients in a field $\F_2$, these exact sequences split.
Hence $\redBetti{Y^{\redProd s} / \tilDelta_s}{q}
       = \redBetti{Y^{\redProd s}}{q} + \redBetti{\tilDelta_s}{q-1}$.
Thus \eqref{eq: InitialDecomp} becomes
\begin{equation} \label{pf: CorPoinSerFormula Eq2}
\begin{aligned}[b]
& \redPoinSer{\Omega(A\redProd \projSp{\infty}{\R}) \union_{A\redProd \projSp{1}{\R}} (Y\redProd \projSp{1}{\R})} \\
 = {}&\sum_{q\ge 0}  \ind^q\sum_{s\ge 1}  \redBetti{Y^{\redProd s}}{q}
         + \sum_{q\ge 0} \ind^{q}\sum_{s\ge 1} \redBetti{\tilDelta_s}{q-1}
\end{aligned}
\end{equation}

Here the first sum in \eqref{pf: CorPoinSerFormula Eq2} is just
\begin{equation} \label{pf: CorPoinSerFormula Eq3}
\begin{aligned}[b]
 \sum_{q\ge 0} \ind^q \sum_{s\ge 1}  \redBetti{Y^{\redProd s}}{q}
 &= \sum_{s\ge 1} \redPoinSer{Y^{\redProd s}} \\
 &= \sum_{s\ge 1} \redPoinSer{Y}^s \\
 &= \frac{\redPoinSer{Y}}{1 - \redPoinSer{Y}}.
  \end{aligned}
\end{equation}
Notice that the geometric series formula is applicable in the last line above.
This is because $Y$ is path-connected so that the constant term of the formal power series $\redPoinSer{Y}$, namely $\redBetti{Y}{0}$, equals zero.

We are left to compute the sum
\begin{equation} \label{eq: defOfS}
\mathcal{S} := \sum_{q\ge 0} \ind^{q}\sum_{s\ge 1} \redBetti{\tilDelta_s}{q-1}
\end{equation}
Replacing $q$ by $q+1$ and noting that $\redBetti{\tilDelta_s}{-1} = 0$, this becomes
\[
\mathcal{S}
= \sum_{q\ge -1} \ind^{q+1}\sum_{s\ge 1} \redBetti{\tilDelta_s}{q}
= \sum_{q\ge 0} \ind^{q+1}\sum_{s\ge 1} \redBetti{\tilDelta_s}{q}
\]
Since $(\redDiag{A})_* : \redH_*(A)\to \redH_*(A\redProd A)$ is the zero map, we may use \eqref{eq: tilDeltaDecomp} to obtain
(noting that $\orbSp{(Y\union_A Y)}{C_2} = Y$ and $\inv{(Y\union_A Y)}{C_2} = A$)
\begin{align*}
\mathcal{S} &= \sum_{q\ge 0} \ind^{q+1}\sum_{s\ge 1}
     \sum_{\substack{|\lambda|+|\mu| = q -s+\dim\lambda+\dim\mu + 1 \\
2\le\dim\lambda+\dim\mu+1\le s}} c_{\lambda,\mu}^{(s)}  \redBetti{Y}{\lambda} \redBetti{A}{\mu} \\
&= \ind \sum_{q\ge 0} \ind^{q}
     \sum_{\substack{|\lambda|+|\mu| = q -s+\dim\lambda+\dim\mu + 1 \\
2\le\dim\lambda+\dim\mu+1\le s}} \redBetti{Y}{\lambda} \redBetti{A}{\mu} \sum_{s\ge 1} c_{\lambda,\mu}^{(s)},
\end{align*}
where
\begin{equation} \label{eq: formulaForCLambdaMu}
c_{\lambda,\mu}^{(s)} = \binom{\dim \lambda +\dim\mu}{\dim\mu}\binom{s-\dim \lambda-\dim\mu-1}{\dim\mu-1}.
\end{equation}
 Solving the equation $\length{\lambda} + \length{\mu} = q - s +\dim{ \lambda} + \dim{\mu} +1$ for the variable $s$
 and using \eqref{eq: formulaForCLambdaMu}, this becomes
\begin{align*}
\mathcal{S}  &= \ind\sum_{q\ge 0} \ind^{q}
\sum_{\substack{2\le\dim\lambda+\dim\mu+1 \\
\dim\lambda+\dim\mu+1\le q+ \dim{\lambda} + \dim{\mu} + 1 - \length{\lambda} - \length{\mu}}}
\redBetti{Y}{\lambda} \redBetti{A}{\mu}  c_{\lambda,\mu}^{(q+ \dim{\lambda} + \dim{\mu} + 1 - \length{\lambda} - \length{\mu})} \\
&=  \ind \sum_{q\ge 0} \ind^{q}
     \sum_{\substack{1\le\dim\lambda+\dim\mu \\
\length{\lambda} +\length{\mu}\le q}} \redBetti{Y}{\lambda} \redBetti{A}{\mu} \binom{\dim \lambda +\dim\mu}{\dim\mu}\binom{q-\length{\lambda} -\length{\mu}}{\dim\mu-1}
\end{align*}

Let $\multInd$ denote the set of multiindexes.
In terms of the Iverson bracket notation \eqref{eq: IversonBracketNotation},
\begin{align*}
\mathcal{S}  &=  \ind \sum_{q\in \Z} \ivBrack{q\ge 0}\ind^{q}
     \sum_{\mu\in\multInd}\sum_{\lambda\in\multInd} \ivBrack{1\le\dim\lambda+\dim\mu} \ivBrack{
\length{\lambda} +\length{\mu}\le q}  \\
& \quad\qquad\cdot\redBetti{Y}{\lambda} \redBetti{A}{\mu}\binom{\dim \lambda +\dim\mu}{\dim\mu}\binom{q-\length{\lambda} -\length{\mu}}{\dim\mu-1} \\
&= \ind \sum_{\mu\in\multInd} \redBetti{A}{\mu}
     \sum_{\lambda\in\multInd}\redBetti{Y}{\lambda} \ivBrack{1\le\dim\lambda+\dim\mu} \binom{\dim \lambda +\dim\mu}{\dim\mu}  \\
&\quad\qquad\cdot\sum_{q\in \Z}\ivBrack{q\ge 0} \ivBrack{
\length{\lambda} +\length{\mu}\le q} \ind^{q}\binom{q-\length{\lambda} -\length{\mu}}{\dim\mu-1} \\
\end{align*}

We compute the innermost sum. The empty multiindex $\emptyset$ has dimension 0. Hence, for multiindexes $\mu$ and $\lambda$,
\begin{align*}
&\sum_{q\in\Z} \ivBrack{q\ge 0}\ivBrack{\length{\lambda} +\length{\mu}\le q}\ind^{q} \binom{q-\length{\lambda} -\length{\mu}}{\dim\mu-1} \\
= {} &\ivBrack{\mu\neq \emptyset}\sum_{q\in\Z} \ivBrack{q\ge 0}\ivBrack{\length{\lambda} +\length{\mu}\le q}\ind^{q} \binom{q-\length{\lambda} -\length{\mu}}{\dim\mu-1} \\
= {}&\ivBrack{\mu\neq \emptyset}\sum_{q\in\Z} \ivBrack{\length{\lambda} +\length{\mu}\le q}\ind^{q}  \binom{q-\length{\lambda} -\length{\mu}}{\dim\mu-1} \\
= {}    & \ivBrack{\mu\neq \emptyset}\ind^{\length{\lambda}+\length{\mu}}\sum_{n\in \Z}  \ind^{n}   \ivBrack{0\le n} \binom{n}{\dim{\mu} - 1} \quad \text{(let $n = q-\length{\lambda} -\length{\mu}$)}\\
= {}   & \ivBrack{\mu\neq \emptyset}\frac{\ind^{\length{\lambda} +\length{\mu} + \dim{\mu} - 1}}{(1-\ind)^{\dim{\mu}}}
\end{align*}
where the last equality follows from the ordinary generating function \eqref{eq: StanGenFuncOfBinomCoeff}.
Hence
\begin{align*}
\mathcal{S}  &=  \ind \sum_{\mu\in\multInd} \redBetti{A}{\mu}
     \sum_{\lambda\in\multInd}\redBetti{Y}{\lambda} \ivBrack{1\le\dim\lambda+\dim\mu} \binom{\dim \lambda +\dim\mu}{\dim\mu} \\
     &\qquad\qquad\qquad\qquad\qquad\qquad\qquad\qquad \cdot \ivBrack{\mu\neq \emptyset}\frac{\ind^{\length{\lambda} +\length{\mu} + \dim{\mu} - 1}}{(1-\ind)^{\dim{\mu}}} \\
&=  \ind \sum_{\mu\in\multInd} \redBetti{A}{\mu}
     \sum_{\lambda\in\multInd}\redBetti{Y}{\lambda}  \binom{\dim \lambda +\dim\mu}{\dim\mu} \ivBrack{\mu\neq \emptyset}\frac{\ind^{\length{\lambda} +\length{\mu} + \dim{\mu} - 1}}{(1-\ind)^{\dim{\mu}}} \\
&=   \sum_{\mu\in\multInd} \redBetti{A}{\mu} \ivBrack{\mu\neq \emptyset} \frac{\ind^{\length{\mu} + \dim{\mu} }}{(1-\ind)^{\dim{\mu}}}
     \sum_{\lambda\in\multInd}\redBetti{Y}{\lambda} \binom{\dim \lambda +\dim\mu}{\dim\mu} \ind^{\length{\lambda}}
\end{align*}

Turning to the next sum, we have
\begin{align*}
&\sum_{\lambda\in\multInd}\redBetti{Y}{\lambda}  \binom{\dim \lambda +\dim\mu}{\dim\mu}
\ind^{\length{\lambda}} \\
= {}&  \sum_{d\in\Z} \ivBrack{d\ge 0}  \binom{d + \dim{\mu} }{\dim{ \mu}}
     \prod_{i=1}^d \sum_{\lambda_i \ge 0} \redBetti{Y}{\lambda_i} \ind^{\lambda_i}       \\
= {}& \sum_{d\in\Z} \ivBrack{d\ge 0} \binom{d + \dim{\mu} }{\dim{ \mu}}
     \redPoinSer{Y}^d \\
= {}& \redPoinSer{Y}^{-\dim{\mu}}
        \sum_{e\in\Z} \ivBrack{e\ge \dim{\mu}}  \binom{e }{\dim{ \mu}}\redPoinSer{Y}^{e} \quad \text{(let $e=d+\dim{\mu}$)} \\
= {}&
        \frac{1}{(1-\redPoinSer{Y})^{\dim{\mu} + 1}}
\end{align*}
where the last equality follows from the ordinary generating function \eqref{eq: StanGenFuncOfBinomCoeff} and
        noting that the formal power series $\redPoinSer{Y}$ has constant term zero.

Hence,
\begin{align*}
 {\mathcal{S}}  &= \sum_{\mu \in \multInd} \redBetti{A}{\mu} \ivBrack{\mu\neq \emptyset}
\frac{\ind^{\length{\mu} + \dim{\mu} }}{(1-\ind)^{\dim{\mu}}} \cdot \frac{1}{(1-\redPoinSer{Y})^{\dim{\mu} + 1}} \\
&=  \sum_{\mu \in \multInd} \redBetti{A}{\mu}
 \ivBrack{\mu\neq\emptyset} \frac{\ind^{\length{\mu} + \dim{\mu} }}{(1-\ind)^{\dim{\mu}}(1-\redPoinSer{Y})^{\dim{\mu} + 1}}
 \\
&=   \sum_{d\ge 1} \frac{\ind^{d}}{(1-\ind)^d(1-\redPoinSer{Y})^{d + 1}} \prod_{i=1}^d \sum_{\mu_i\ge 0} \redBetti{A}{\mu_i} t^{\mu_i} \\
&=   \sum_{d\ge 1} \frac{\ind^{d}}{(1-\ind)^d(1-\redPoinSer{Y})^{d + 1}}\redPoinSer{Y}^d \\
&= \frac{\ind\redPoinSer{A}}{(1-\redPoinSer{Y})\left((1-\ind)(1-\redPoinSer{Y})  - \ind\redPoinSer{A}\right)},
\end{align*}
where the last equality follows from the geometric series formula and the observation that the power series $\redPoinSer{Y}$ has constant term zero.
Therefore, combining this with equations \eqref{pf: CorPoinSerFormula Eq2}, \eqref{pf: CorPoinSerFormula Eq3} and \eqref{eq: defOfS},
    we obtain the required
\begin{align*}
 &\redPoinSer{\Omega(A\redProd \projSp{\infty}{\R}) \union_{A\redProd \projSp{1}{\R}} (Y\redProd \projSp{1}{\R})} \\
= {}  & \frac{\redPoinSer{Y}}{1- \redPoinSer{Y}}
         +  \frac{\ind\redPoinSer{A}}{(1-\redPoinSer{Y})\left(1-\ind - (1 - \ind)\redPoinSer{Y} - \ind\redPoinSer{A}\right)} \\
= {}  & \frac{                (1-\ind)\redPoinSer{Y} + \ind \redPoinSer{A} }
              {   1 - \ind - (1-\ind)\redPoinSer{Y} - \ind\redPoinSer{A}  }
\end{align*}
\end{proof}

\section{Examples} \label{sec: examples}

In this section, we give some examples of theorem \ref{thm: PoinSer} and explain the relation to a spectral sequence studied by Bousfield-Curtis.

The following is theorem 10.2 of \cite{BousfieldCurtis70}.
Recall our convention that homology is taken modulo 2.
Let $X$ be a reduced simplicial set.
Let $GX$ be the simplicial group which is Kan's construction \cite{Kan58}, whose underlying space is $\Loop X$.
Filter the group ring $\F_2 (GX)$ by powers of the augmentation ideal
\[
\cdots \subset I^{n+1} \subset I^n \subset \cdots \subset I^1 \subset I^0 = \F_2 (GX)
\]
The associated spectral sequence $\{\bar{E}^r, d^r:\bar{E}^r_{s,t} \to \bar{E}^r_{s+r,t-1}\}$ has the following properties:
\begin{enumerate}[(a)]
\item For $r\ge 1$, the $\bar{E}^r$ term is a differential graded Hopf algebra.
\item The $\bar{E}^1$ term is given by $\bar{E}^1_{s,t} = \pi_t(I^s/I^{s+1})$.
As algebras, the $\bar{E}^1$ term is isomorphic to the tensor algebra $T(s^{-1}\redH_*(X))$.
Here, for a graded vector space $V_*$, its {\emph{desuspension}} $(s^{-1}V)_*$ is the graded vector space given by $(s^{-1}V)_{n} = V_{n+1}$ for all $n\ge 0$.
The differential $d^1$ on $\bar{E}^1_{1,*}$ is given by the comultiplication of $\redH_*(X)$.
\item If $X$ is simply-connected, then the spectral sequence $\{\bar{E}^r, d^r\}$ converges to $H_*(G X)$.
            Thus $\bar{E}^\infty$ is the graded Hopf algebra associated with a decreasing filtration of $H_*(G X) \isom H_*(\Loop X)$.
\end{enumerate}

For $V_*$ a graded vector space, let $\chi(V_*) := \sum_{q\ge 0} \ind^q \dim V_q$ denote its Euler-Poincar\'{e} series.
We compute the Euler-Poincar\'{e} series of the $\bar{E}^1$ and $\bar{E}^\infty$ terms.
Note that
\[
\chi(T(s^{-1}\redH_*(X))) = \frac{1}{1- \chi(s^{-1}\redH_*(X))} = \frac{1}{1 - \ind^{-1}\redPoinSer{X}} = \frac{\ind}{\ind - \redPoinSer{X}}
\]
Hence, by (b),
\begin{equation}\label{eq: E1EulerSeries}
\chi(\bar{E}^1) = \frac{\ind}{\ind - \redPoinSer{X}}
\end{equation}
Furthermore, if $X$ is simply-connected, then by (c),
\begin{equation} \label{eq: EInftyEulerSeries}
\chi(\bar{E}^\infty) = \chi(H_*(\Loop X)) = \chi(\redH_*(\Loop X)) + 1 = \redPoinSer{\Loop X} + 1
\end{equation}

We give a criterion for the spectral sequence $\{\bar{E}^r,d^r\} $ to collapse at the $\bar{E}^1$ term.
\begin{proposition} \label{prop: BousfieldCurtis}
Let $X$ be a pointed space whose mod 2 homology groups are finite-dimensional.
Suppose that $X$ is simply-connected and the map $(\redDiag{X})_*: \redH_*(X)\to\redH_*( X\redProd X)$ induced by the reduced diagonal of $X$ is the zero map.
Then the spectral sequence $\{\bar{E}^r,d^r\} $ collapses at the $\bar{E}^1$ term and
\begin{equation} \label{eq: BousfieldCurtisEq}
\redPoinSer{\Loop X} = \frac{\redPoinSer{X}}{\ind - \redPoinSer{X} }
\end{equation}
\end{proposition}
\begin{proof}
The comultiplication on $\redH_*(X)$ is induced by the reduced diagonal map $\redDiag{X} : X \to X\redProd X$.
Hence, by (b) above, if $(\redDiag{X})_*: \redH_*(X)\to\redH_*( X\redProd X)$ is the zero map, then $d^1 = 0$, so that $\bar{E}^1 = \bar{E}^\infty$.
Furthermore, if $X$ is simply-connected, then using \eqref{eq: E1EulerSeries} and \eqref{eq: EInftyEulerSeries} and comparing Euler-Poincar\'{e} series,
we have
\[
 \frac{\ind}{\ind - \redPoinSer{X}} = \redPoinSer{\Loop X} + 1
\]
Equation \eqref{eq: BousfieldCurtisEq} follows by subtracting $1$ from both sides.
\end{proof}
As noted in the introduction, this proposition holds when $X$ is a simply-connected co-H-space.
This gives yet another proof of \eqref{eq: redPoinSerOfLoopSus} when $K= \F_2$.
In more detail, let a path-connected pointed space $Y$ be given.
Then $\sus Y$ is a simply-connected co-H-space, so by proposition \ref{prop: BousfieldCurtis},
\begin{align*}
\redPoinSer{\Loop \sus X} &= \frac{\redPoinSer{\sus Y}}{\ind - \redPoinSer{\sus Y}} \\
&= \frac{\ind\redPoinSer{Y}}{\ind - \ind\redPoinSer{ Y}} \\
&= \frac{\redPoinSer{Y}}{1 - \redPoinSer{Y}}
\end{align*}

Theorem \ref{thm: PoinSer} has the following consequence which relates to this spectral sequence $\{\bar{E}^r,d^r\} $ studied by Bousfield-Curtis.
\begin{corollary} \label{cor: hardWork}
Let $A\inc Y$ be a based inclusion of pointed spaces,
    both of whose mod 2 homology groups are finite-dimensional.
Suppose that this inclusion induces a monomorphism $\redH_*(A) \rightarrowtail \redH_*(Y)$ in reduced homology.
Suppose further that $Y$ is path-connected
and the map $(\redDiag{A})_* : \redH_*(A)\to \redH_*(A\redProd A)$ induced by the reduced diagonal map of $A$ is the zero map.
Then the spectral sequence $\{\bar{E}^r,d^r\} $ collapses at the $\bar{E}^1$ term.
\end{corollary}

\begin{proof}
As we are taking homology with coefficients in $\F_2$, the inclusion $\projSp{1}{\R}\inc\projSp{\infty}{\R}$
    induces a monomorphism $\redH_*(\projSp{1}{\R}) \rightarrowtail \redH_*(\projSp{\infty}{\R})$ in reduced homology.
Since the inclusion $A\inc Y$ induces a monomorphism $\redH_*(A) \rightarrowtail \redH_*(Y)$ in reduced homology,
the induced maps $\phi:\redH_*(A\redProd \projSp{1 }{\R}) \to \redH_*(Y\redProd \projSp{1 }{\R})$
    and $\psi: \redH_*(A\redProd \projSp{1}{\R}) \to \redH_*(A\redProd \projSp{\infty}{\R})$ are also monomorphisms.
This implies that the Mayer-Vietoris long exact sequence associated to the union $X:=(A\redProd \projSp{\infty}{\R}) \union_{A\redProd \projSp{1}{\R}} (Y\redProd \projSp{1}{\R})$
    splits into short exact sequences
    \[
0\to \redH_q(A\redProd \projSp{1}{\R}) \xrightarrow{(\phi,\psi)} \redH_q(Y\redProd \projSp{1}{\R}) \oplus \redH_q(A\redProd \projSp{\infty}{\R}) \to \redH_q(X) \to 0
    \]
Since we are taking coefficients in a field $\F_2$, these short exact sequences split. Hence
\begin{align*}
\redPoinSer{X} &= \redPoinSer{Y\redProd \projSp{1}{\R}} + \redPoinSer{A\redProd \projSp{\infty}{\R}} - \redPoinSer{A\redProd \projSp{1}{\R}} \\
&= \ind \redPoinSer{Y} + \frac{\ind}{1-\ind}\redPoinSer{A} - \ind\redPoinSer{A} \\
&= \ind \redPoinSer{Y} + \frac{\ind^2}{1-\ind}\redPoinSer{A}
\end{align*}
Substituting this into \eqref{eq: E1EulerSeries} gives
\begin{align*}
\chi(\bar{E}^1)
&= \frac{\ind}{\ind - \left(\ind \redPoinSer{Y} + \frac{\ind^2}{1-\ind}\redPoinSer{A}\right)}   \\
&= \frac{                1-\ind }
              {   1 - \ind - (1-\ind)\redPoinSer{Y} - \ind\redPoinSer{A}  }
\end{align*}

Since $Y$ is path-connected, the smash product $Y\redProd \projSp{1}{\R}$ is simply-connected.
Since $(\redDiag{A})_* : \redH_*(A)\to \redH_*(A\redProd A)$ is the zero map,
     $A$ is path-connected so that $A \redProd \projSp{\infty}{\R}$ is simply-connected.
Hence the van-Kampen theorem implies that $X$ is simply-connected.
Thus we may use \eqref{eq: EInftyEulerSeries} together with theorem \ref{thm: PoinSer} to obtain
\begin{align*}
\chi(\bar{E}^\infty) &=  \redPoinSer{\Loop X} + 1 \\
&= \frac{                (1-\ind)\redPoinSer{Y} + \ind \redPoinSer{A} }
              {   1 - \ind - (1-\ind)\redPoinSer{Y} - \ind\redPoinSer{A}  } + 1 \\
              &=  \frac{                1-\ind }
              {   1 - \ind - (1-\ind)\redPoinSer{Y} - \ind\redPoinSer{A}  }
\end{align*}
Thus $\chi(\bar{E}^1) = \chi(\bar{E}^\infty)$.
Hence in fact $\bar{E}^1=\bar{E}^\infty$,
that is to say, the spectral sequence $\{\bar{E}^r,d^r\} $ collapses at the $\bar{E}^1$ term.
\end{proof}

We now bring three extremal special cases of theorem \ref{thm: PoinSer} to the attention of the reader.

\begin{example} \label{ex: loopSus}
Let $Y$ be a path-connected pointed space.

Taking $A = \pt$ in theorem \ref{thm: PoinSer}, we have
\[
            \redPoinSer{\Loop\sus Y}
                = \frac{                \redPoinSer{Y} }{   1  - \redPoinSer{Y}  }
            \]
This gives yet another proof of \eqref{eq: redPoinSerOfLoopSus} in the case where $K = \F_2$.
Furthermore, corollary \ref{cor: hardWork} implies that the spectral sequence $\{\bar{E}^r,d^r\} $ collapses at the $\bar{E}^1$ term,
    which agrees with proposition \ref{prop: BousfieldCurtis}.
\end{example}

\begin{example} \label{ex: wuExample}
Let $A$ be a pointed space such that the map $(\redDiag{A})_* : \redH_*(A)\to \redH_*(A\redProd A)$ induced by its reduced diagonal map is zero.
In particular, $A$ is path-connected.

Hence, taking $A= Y$ in theorem \ref{thm: PoinSer}, we have
\[
            \redPoinSer{\Loop(A\redProd \projSp{\infty}{\R})}
                = \frac{                \redPoinSer{A} }{   1 - \ind - \redPoinSer{A}  }
            \]
Furthermore, corollary \ref{cor: hardWork} implies that the spectral sequence $\{\bar{E}^r,d^r\} $ collapses at the $\bar{E}^1$ term.
This agrees with proposition \ref{prop: BousfieldCurtis}.
Indeed, since $A$ is path-connected, $A\redProd \projSp{\infty}{\R}$ is simply-connected.
Also, since $(\redDiag{A})_* : \redH_*(A)\to \redH_*(A\redProd A)$ is the zero map, $(\redDiag{A\redProd \projSp{\infty}{\R}})_* : \redH_*(A\redProd \projSp{\infty}{\R})\to \redH_*((A\redProd \projSp{\infty}{\R})\redProd (A\redProd \projSp{\infty}{\R}))$ is zero.
Thus the conditions in proposition \ref{prop: BousfieldCurtis} hold.
\end{example}

\begin{example}
Let $A$ be a pointed space such that the map $(\redDiag{A})_* : \redH_*(A)\to \redH_*(A\redProd A)$ induced by its reduced diagonal map is zero.

Hence, taking $Y = CA$ to be the cone of $A$ in theorem \ref{thm: PoinSer} and noting that $CA$ is contractible, we have
\[
            \redPoinSer{\Loop(A\redProd (\projSp{\infty}{\R}/\projSp{1}{\R}))}
                = \frac{                \ind\redPoinSer{A} }{   1 - \ind - \ind\redPoinSer{A}  }
\]
In this case, corollary \ref{cor: hardWork} does not apply.
However, we know from proposition \ref{prop: BousfieldCurtis} and an argument similar to example \ref{ex: wuExample},
that the spectral sequence $\{\bar{E}^r,d^r\} $ does collapse at the $\bar{E}^1$ term.
\end{example}

We end this article by giving an example of theorem \ref{thm: PoinSer} not covered by the above examples.
In the process, we prove a conjecture of the first and third authors (see conjecture 6.1 in \cite{GaoWu13}).
\begin{example} \label{ex: LangExample}
Let $X$ be the pointed $C_2$-space
    which is the union of  two 2-spheres $S^2$ with the antipodal involution under which their equatorial circles are identified.
As noted in \cite{GaoWu13},
    this pointed $C_2$-space $X$ is equivariantly homotopic to the pointed $C_2$-space $S^2 \cup_{S^1} S^2$
        with the action associated to the based involution of switching the two copies of $S^2$.
Hence, by lemma \ref{lem: unionAsStuntRedBor} and taking $A=S^1$ and $Y= S^2$ in theorem \ref{thm: PoinSer},
\begin{equation} \label{eq: LangExample}
\begin{aligned}
\redPoinSer{\Loop (\stuntRedBor{X}{C_2})}
    &= \frac{
                (1-\ind)(\ind^2) +\ind (\ind)
                }
                {
                1 - \ind - (1-\ind)(\ind^2) - \ind(\ind)
                } \\
    &= \frac{1-\ind}{\ind^3 - 2\ind^2 - \ind + 1} - 1,
\end{aligned}
\end{equation}
which proves conjecture 6.1 in \cite{GaoWu13}.
Since the induced map $\redH_*(S^1) \to \redH_*(S^2)$ is not a monomorphism, corollary \ref{cor: hardWork} does not apply.

From the generating function \eqref{eq: LangExample},
    the mod 2 reduced Betti numbers of $\Loop (\stuntRedBor{X}{C_2})$ form the sequence
$$\{0,2,1,5,5,14,19,42,66,131,221,417,\ldots \}$$
As mentioned in the introduction, this is essentially sequence A052547 in the On-Line Encyclopedia of Integer Sequences \cite{OEIS}.
This sequence is related to the Chebyshev polynomials used in approximation theory.
\end{example}
This example suggests that the mod 2 reduced Betti numbers of the loop space studied in theorem \ref{thm: PoinSer}
    may provide a geometric interpretation for certain combinatorial sequences.
This application to combinatorics may be of interest to pursue further.

{\textbf{Acknowledgements.}}
The first author thanks the organizers
    of the 2014 ICM satellite conference in algebraic topology in Dalian, China
    for the opportunity to speak about this material.
The second author thanks the organizers of the 5th East Asian conference in algebraic topology held in Beijing, China in 2013
    for the opportunity to speak about preliminary versions of this material.

\end{document}